\documentclass[12pt,english,refpage, intoc, refeq]{article}
\usepackage[T1]{fontenc}
\usepackage[latin9]{inputenc}
\usepackage{xcolor, pdfcolmk, babel, units, mathtools,amsthm,amssymb,graphicx}
\PassOptionsToPackage{normalem}{ulem}
\usepackage{ulem}
\usepackage[unicode=true,pdfusetitle,
 bookmarks=true,bookmarksnumbered=false,bookmarksopen=false,
 breaklinks=false,pdfborder={0 0 1},backref=page,colorlinks=true]
 {hyperref}
\hypersetup{
 linkcolor=blue, citecolor=blue, urlcolor=blue, filecolor=blue}

\makeatletter

\providecolor{lyxadded}{rgb}{0,0,1}
\providecolor{lyxdeleted}{rgb}{1,0,0}

\DeclareRobustCommand{\lyxsout}[1]{\ifx\\#1\else\sout{#1}\fi}

\numberwithin{equation}{section}
\numberwithin{figure}{section}
\theoremstyle{plain}
\newtheorem{thm}{\protect\theoremname}
  \theoremstyle{plain}
  \newtheorem{lem}[thm]{\protect\lemmaname}
  \theoremstyle{remark}
  \newtheorem{rem}[thm]{\protect\remarkname}

\makeatother

  \providecommand{\lemmaname}{Lemma}
  \providecommand{\remarkname}{Remark}
\providecommand{\theoremname}{Theorem}

\begin{document}

\title{The Stein Characterization of $M$-Wright Distributions}

\author{\textbf{Jos{\'e} Lu{\'\i}s da Silva},\\
CCM, University of Madeira, Campus da Penteada,\\
9020-105 Funchal, Portugal.\\
Email: luis@uma.pt\and \textbf{Mohamed Erraoui}\\
Universit{\'e} Cadi Ayyad, Facult{\'e} des Sciences Semlalia,\\
 D{\'e}partement de Math{\'e}matiques, BP 2390, Marrakech, Maroc\\
Email: erraoui@uca.ma}
\maketitle
\begin{abstract}
In this paper use the Stein method to characterize the $M$-Wright
distribution $M_{\nicefrac{1}{3}}$ and its symmetrization. The Stein
operator is associated with the general Airy equation and the corresponding
Stein equation is nothing but a general inhomogeneous Airy equation.
\end{abstract}

\section{Introduction}

Stein's method is a powerful technic used for studying
approximations of probability distributions, and is best known for
its ability to establish convergence rates. It was initially conceived
by Charles Stein in the seminal paper \cite{Stein1972} to provide
errors in the approximation by the normal distribution of the distribution
of the sum of dependent random variables of a certain structure. However,
the ideas presented are sufficiently powerful to be able to work well
beyond that intended purpose, applying to approximation of more general
random variables by distributions other than the normal (such as the
Poisson, exponential, Gamma, etc). We refer to \cite{Ross2011,Barbour:2005ur,CGS2011,Stein:1986tz},
and references therein for more details on this method. 

The $M$-Wright probability density function $M_{\beta}$, $0<\beta<1$
defined on $\mathbb{R}_{+}$, was introduced by F.\ Mainardi in order
to study time-fractional diffusion-wave equation, see \cite{Mainardi1994a}.
The function $M_{\beta}$ is a special case of the Wright function,
see \cite{Mainardi_Mura_Pagnini_2010} for more details. The distribution
$\nu_{\beta}$ on $\mathbb{R}_{+}$ with density $M_{\beta}$ with
respect to the Lebesgue measure, i.e., $d\nu_{\beta}(x)=M_{\beta}(x)\,dx$,
$x\ge0$, we call $M$-Wright distribution. Its Laplace transform
is given by
\[
\int_{0}^{\infty}e^{-xt}M_{\beta}(x)\,dx=E_{\beta}(-t),\quad\forall t\ge0,
\]
where $E_{\beta}$ is the Mittag-Leffler function, cf.\ \cite{GKMS2014}
for details, defined by
\begin{equation}
E_{\beta}(z):=\sum_{n=0}^{\infty}\frac{z^{n}}{\Gamma(\beta n+1)},\qquad z\in\mathbb{C}.\label{eq:Mittag-Leffler-function}
\end{equation}
The density $M_{\nicefrac{1}{q}}$, $q=2,3,\ldots$ satisfies the
ODE of order $q-1$, cf.\ \cite{Mainardi1994} 
\begin{equation}
\frac{d^{q-1}}{dx^{q-1}}M_{\nicefrac{1}{q}}(x)+\frac{(-1)^{q}}{q}xM_{\nicefrac{1}{q}}(x)=0,\quad x\ge0.\label{eq:general_MWf_ode}
\end{equation}
For the special cases $q=2$ and $q=3$, we have
\begin{eqnarray}
M_{\nicefrac{1}{2}}(x) & = & \frac{1}{\sqrt{\pi}}\exp\left(-\frac{x^{2}}{4}\right),\nonumber \\
M_{\nicefrac{1}{3}}(x) & = & 3^{\nicefrac{2}{3}}\mathrm{Ai}\left(\frac{x}{3^{\nicefrac{1}{3}}}\right),\label{eq:MWright_and_Airy}
\end{eqnarray}
where $\mathrm{Ai}$ is the Airy function, see Appendix\ \ref{sec:special_functions}.
The density $M_{\nicefrac{1}{2}}$ was characterized in the pioneering
work of Stein \cite{Stein1972}, see also \cite{Stein:1986tz} for
more details. 

In this paper we investigate the Stein method for the distribution
$M_{\nicefrac{1}{3}}$ as well as its symmetrization $\hat{M}_{\nicefrac{1}{3}}$
and characterize them via the Stein equation, see Theorems\ \ref{thm:charact_MW_dist}
and \ref{thm:charact_MW_sym} below. The distribution $\hat{M}_{\nicefrac{1}{3}}$
is squeezed between the Gaussian and the symmetric Laplacian distributions,
see Figure\ \ref{fig:Plots-MWright}, which were characterized in
\cite{Stein:1986tz} and \cite{Pike2014}, respectively. In Section\ \ref{sec:Stein-MWright}
we characterize the distribution $M_{\nicefrac{1}{3}}$ and in Section\ \ref{sec:Stein_MWright-sym}
we do the same for the symmetric case $\hat{M}_{\nicefrac{1}{3}}$. 

It follows from equation (\ref{eq:general_MWf_ode}), with $q=3$,
that $M_{\nicefrac{1}{3}}$ solves the Airy differential equation
\begin{equation}
y''(x)-\frac{1}{3}xy(x)=0,\quad x\ge0.\label{eq:Airy_ODE}
\end{equation}
It is well known that the two independent solutions of equation (\ref{eq:Airy_ODE})
are the Airy functions of the first and second kind, $\mathrm{Ai}$
and $\mathrm{Bi}$, respectively, see Appendix\ \ref{sec:special_functions}
for more details on these functions. The first step in extending the
Stein method for the distribution $M_{\nicefrac{1}{3}}$ is to find
a proper Stein equation from (\ref{eq:Airy_ODE}). This leads us to
consider as Stein's equation the inhomogeneous Airy differential equation
\begin{equation}
y''(x)-\frac{1}{3}xy(x)=h(x)-\mathbb{E}\big(h(Y)\big),\quad x\ge0,\label{eq:Stein_equation0}
\end{equation}
where $Y$ is a random variable with density $M_{\nicefrac{1}{3}}$
and $h$ is a suitable function, see Lemma\ \ref{lem:Stein_solution}
below for details. The solution of equation (\ref{eq:Stein_equation0})
is obtained from the two linearly independent solutions $\mathrm{Ai}$
and $\mathrm{Bi}$ of equation (\ref{eq:Airy_ODE}) using the method
of variation of parameters. We show that for any real-valued continuous
bounded function $h$ on $\mathbb{R}_{+}$ the solution of equation\ (\ref{eq:Stein_equation0})
belongs to $C_{b}^{2}(\mathbb{R}_{+})$ (the space of bounded twice
continuously differentiable functions on $\mathbb{R}_{+}$ with bounded
derivatives), see Lemma\ \ref{thm:charact_MW_sym}. Once this done
we characterize $M_{\nicefrac{1}{3}}$ via the second order Stein's
operator 
\begin{equation}
(\mathcal{A}_{\nicefrac{1}{3}}f)(x):=f''(x)-\frac{1}{3}xf(x),\quad x\ge0,\label{eq:Stein-Operator}
\end{equation}
for a sufficiently smooth function $f$, see Section\ \ref{sec:Stein-MWright}
for details. This is the contents of Section\ \ref{sec:Stein-MWright}. 

In Section\ \ref{sec:Stein_MWright-sym} we apply the above scheme
to the symmetric density $\hat{M}_{\nicefrac{1}{3}}$ on $\mathbb{R}$.
It turns out that $\hat{M}_{\nicefrac{1}{3}}$ solves the ODE 
\begin{equation}
y''(x)-\frac{1}{3}|x|y(x)=0,\quad x\in\mathbb{R},\label{eq:Stein_equation_sym-1}
\end{equation}
and as Stein's equation we consider 
\begin{equation}
y''(x)-\frac{1}{3}|x|y(x)=\hat{h}(x),\label{eq:Stein-equation-sym-nonhom}
\end{equation}
where $\hat{h}$ is defined 
\[
\hat{h}(x):=\big[h(x)-\mathbb{E}\big(h(Y)\big)\big]1\!\!1_{[0,\infty)}(x)+\big[h(x)-\mathbb{E}\big(h(-Y)\big)\big]1\!\!1_{(-\infty,0)}(x).
\]
Notice that in general $\hat{h}$ is not continuous at $x=0$ which
implies less regularity of the solution of equation\ (\ref{eq:Stein-equation-sym-nonhom}).
More precisely, in Lemma\ \ref{lem:Solution_Stein_sym} we show that
the solution $f$ of equation\ (\ref{eq:Stein-equation-sym-nonhom})
is such $f\in C_{b}^{1}(\mathbb{R})$ and $f''\in C_{b}(\mathbb{R}^{*})$,
see also Remark\ \ref{rem:Stein-sym-h} for more details. Here and
below $C_{b}^{k}(X)$ denotes the Banach space of bounded $k$th continuous
differentiable functions $f:X\longrightarrow\mathbb{R}$ endowed with
the supremum norm $\|\cdot\|_{\infty}$.

\section{The Stein Characterization of $M_{\nicefrac{1}{3}}$ Distribution}

\label{sec:Stein-MWright}In this section we use the Stein method
to characterize the distribution $M_{\nicefrac{1}{3}}$. To this end,
at first we introduce the functional framework. Define the space $\mathcal{D}_{\nicefrac{1}{3}}$
of functions in $C_{b}^{2}(\mathbb{R}_{+})$ such that 
\begin{equation}
\frac{f'(0)}{\Gamma(\nicefrac{2}{3})}-\frac{f(0)}{\Gamma(\nicefrac{1}{3})}=0,\label{eq:boundary_condition}
\end{equation}
and consider the operator $(\mathcal{A}_{\nicefrac{1}{3}},\mathcal{D}_{\nicefrac{1}{3}})$. 

The following theorem states the main result of this section.
\begin{thm}[Characterization of $M_{\nicefrac{1}{3}}$]
\label{thm:charact_MW_dist}Let $X$ be a positive random variable.
Then $X$ follows the $M_{\nicefrac{1}{3}}$ distribution if and only
if 
\[
\mathbb{E}\big((\mathcal{A}_{\nicefrac{1}{3}}f)(X)\big)=0,\quad\forall f\in\mathcal{D}_{\nicefrac{1}{3}}.
\]
\end{thm}

Before proving Theorem\ \ref{thm:charact_MW_dist} we need a technical
lemma.
\begin{lem}
\noindent \label{lem:Stein_solution}Let $h:\mathbb{R}_{+}\longrightarrow\mathbb{R}$
be a bounded continuous function and denote by $\tilde{h}(x):=h(x)-\mathbb{E}\big(h(Y)\big)$
where $Y$ has $M_{\nicefrac{1}{3}}$ distribution. Then the function
\begin{equation}
f_{h}(x)=-3^{\nicefrac{1}{3}}\pi\left[\mathrm{Ai}\left(\frac{x}{3^{\nicefrac{1}{3}}}\right)\int_{0}^{x}\mathrm{Bi}\left(\frac{t}{3^{\nicefrac{1}{3}}}\right)\tilde{h}(t)\,dt+\mathrm{Bi}\left(\frac{x}{3^{\nicefrac{1}{3}}}\right)\int_{x}^{\infty}\mathrm{Ai}\left(\frac{t}{3^{\nicefrac{1}{3}}}\right)\tilde{h}(t)\,dt\right],\label{eq:Stein_solution}
\end{equation}
 solves the Stein equation (\ref{eq:Stein_equation0}). In addition,
there exists non-negative constants $\tilde{C}_{1}$, $\tilde{C}_{2}$
and $\tilde{C}_{3}$ such that $\|f_{h}\|_{\infty}\le\tilde{C}_{1}\|\tilde{h}\|_{\infty}$,
$\|f_{h}'\|_{\infty}\le\tilde{C}_{2}\|\tilde{h}\|_{\infty}$, $\|f_{h}''\|_{\infty}\le\tilde{C}_{3}\|\tilde{h}\|_{\infty}$
and $f_{h}$ belongs to $\mathcal{D}_{\nicefrac{1}{3}}$. 
\end{lem}

\begin{proof}
\noindent First consider $w_{1}$ and $w_{2}$ the two independent
solutions of the corresponding homogeneous Stein equation (\ref{eq:Stein_equation0}).
They are given in terms of the Airy functions
\begin{eqnarray}
w_{1}(x) & = & M_{\nicefrac{1}{3}}(x)=3^{\nicefrac{2}{3}}\mathrm{Ai}\left(\frac{x}{3^{\nicefrac{1}{3}}}\right),\quad x\ge0,\label{eq:Airy1}\\
w_{2}(x) & = & 3^{\nicefrac{2}{3}}\mathrm{Bi}\left(\frac{x}{3^{\nicefrac{1}{3}}}\right),\quad x\ge0.\label{eq:Airy2}
\end{eqnarray}
The solution of the inhomogeneous Stein equation is given using the
method of the variation of the parameters, see for example \cite{Collins2007},
in explicit
\begin{equation}
f_{h}(x)=-w_{1}(x)\int_{0}^{x}\frac{w_{2}(t)\tilde{h}(t)}{W(t)}\,dt-w_{2}(x)\int_{x}^{\infty}\frac{w_{1}(t)\tilde{h}(t)}{W(t)}\,dt,\label{eq:Stein_gsolution}
\end{equation}
where $W$ is the Wronskian of the solutions $w_{1},w_{2}$ equal
to $\nicefrac{3}{\pi}$. Hence, the solution (\ref{eq:Stein_solution})
results from (\ref{eq:Stein_gsolution}) using (\ref{eq:Airy1})-(\ref{eq:Stein_gsolution}).
As $h$ is a continuous function, then it follows that $f_{h}$ is
twice continuously differentiable. Moreover, for any $x\geq0$, we
have
\begin{eqnarray*}
|f_{h}(x)| & \leq & 3^{\nicefrac{1}{3}}\pi\|\tilde{h}\|_{\infty}\left(\mathrm{Ai}\left(\frac{x}{3^{\nicefrac{1}{3}}}\right)\int_{0}^{x}\mathrm{Bi}\left(\frac{t}{3^{\nicefrac{1}{3}}}\right)\,dt+\mathrm{Bi}\left(\frac{x}{3^{\nicefrac{1}{3}}}\right)\int_{x}^{\infty}\mathrm{Ai}\left(\frac{t}{3^{\nicefrac{1}{3}}}\right)\,dt\right)\\
 & = & 3^{\nicefrac{2}{3}}\pi\|\tilde{h}\|_{\infty}\left(\mathrm{Ai}\left(\frac{x}{3^{\nicefrac{1}{3}}}\right)\int_{0}^{\frac{x}{3^{\nicefrac{1}{3}}}}\mathrm{Bi}(t)\,dt+\mathrm{Bi}\left(\frac{x}{3^{\nicefrac{1}{3}}}\right)\int_{\frac{x}{3^{\nicefrac{1}{3}}}}^{\infty}\mathrm{Ai}(t)\,dt\right)\\
 & \leq & 3^{\nicefrac{2}{3}}\pi\|\mathrm{Gi}\|_{\infty}\|\tilde{h}\|_{\infty},
\end{eqnarray*}
where $\mathrm{Gi}$ is the Scorer function, see Appendix\ \ref{sec:special_functions},
equation\ (\ref{eq:ScorerGi-function}). 

\noindent Now differentiating (\ref{eq:Stein_solution}) we obtain
\[
f'_{h}(x)=-\pi\bigg[\mathrm{Ai}'\left(\frac{x}{3^{\nicefrac{1}{3}}}\right)\int_{0}^{x}\mathrm{Bi}\left(\frac{t}{3^{\nicefrac{1}{3}}}\right)\tilde{h}(t)\,dt+\mathrm{Bi}'\left(\frac{x}{3^{\nicefrac{1}{3}}}\right)\int_{x}^{\infty}\mathrm{Ai}\left(\frac{t}{3^{\nicefrac{1}{3}}}\right)\tilde{h}(t)\,dt\bigg]
\]
and estimating as above implies, using equation\ (\ref{eq:Scorer_deriv_asymp}),
yields
\[
|f_{h}'(x)|\leq3^{\nicefrac{1}{3}}\pi\|\mathrm{Gi}'\|_{\infty}\|\tilde{h}\|_{\infty}.
\]
Since $f_{h}$ is the solution of the Stein equation (\ref{eq:Stein_equation0}),
then
\[
|f''_{h}(x)|\leq\frac{1}{3}|xf_{h}(x)|+\|\tilde{h}\|_{\infty}\le3^{-\nicefrac{2}{3}}\pi|x\mathrm{Gi}(x)|\|\tilde{h}\|_{\infty}+\|\tilde{h}\|_{\infty},\quad\forall x\in\mathbb{R}_{+}.
\]
As the function $\mathbb{R}_{+}\ni x\mapsto x\mathrm{Gi}(x)\in\mathbb{R}$
is bounded, cf. Appendix\ \ref{sec:special_functions} equation\ (\ref{eq:ScorerGi_asymptotic}),
it follows that $f_{h}''$ is bounded and we have
\[
|f''_{h}(x)|\leq\big(3^{-\nicefrac{2}{3}}\pi\sup_{x\in\mathbb{R}_{+}}|x\mathrm{Gi}(x)|+1\big)\|\tilde{h}\|_{\infty}.
\]
Finally, to see that $f_{h}$ satisfies equation (\ref{eq:boundary_condition})
we notice that $\mathrm{Bi}(0)=\frac{1}{3^{\nicefrac{1}{6}}\Gamma(\nicefrac{2}{3})}$
and $\mathrm{Bi}'\left(0\right)=\frac{3^{\nicefrac{1}{6}}}{\Gamma(\nicefrac{1}{3})}$.
\end{proof}

\begin{proof}[Proof of Proposition\ \ref{thm:charact_MW_dist}.]
\emph{Necessity}. Let $Y$ be a random variable with $M_{\nicefrac{1}{3}}$
distribution and $f\in\mathcal{D}_{\nicefrac{1}{3}}$. Then, it is
clear that
\[
\mathbb{E}\big((\mathcal{A}f)(Y)\big)=\int_{0}^{\infty}\big(f''(x)-\frac{1}{3}xf(x)\big)M_{\nicefrac{1}{3}}(x)\,dx
\]
and an integration by parts yields
\begin{eqnarray*}
\mathbb{E}\big((\mathcal{A}f)(Y)\big) & = & \int_{0}^{\infty}\big(M''_{\nicefrac{1}{3}}(x)-\frac{1}{3}xM_{\nicefrac{1}{3}}(x)\big)f(x)\,dx\\
 &  & +f'(x)M_{\nicefrac{1}{3}}(x)\big|_{0}^{\infty}-f(x)M'_{\nicefrac{1}{3}}(x)\big|_{0}^{\infty}.
\end{eqnarray*}
Taking into account that $M_{\nicefrac{1}{3}}$ satisfies (\ref{eq:Airy_ODE}),
$M_{\nicefrac{1}{3}}(0)=\nicefrac{1}{\Gamma(\nicefrac{2}{3})}$, $M'_{\nicefrac{1}{3}}(0)=-\nicefrac{1}{\Gamma(\nicefrac{1}{3})}$
and the asymptotic behavior for $M_{\nicefrac{1}{3}}$, $M'_{\nicefrac{1}{3}}$
in terms of the Airy function, see equation\ (\ref{eq:MWright_and_Airy})
and Appendix\ \ref{sec:special_functions}, equation\ (\ref{eq:AiryAi-asymptotic-positiv}),
we obtain $\mathbb{E}\big((\mathcal{A}f)(Y)\big)=0$.

\noindent \emph{Sufficiency}. Let $h$ be a continuous bounded function.
By Lemma\ \ref{lem:Stein_solution}, the function $f_{h}\in\mathcal{D}_{\nicefrac{1}{3}}$
and satisfies the Stein equation (\ref{eq:Stein_equation0}): 
\[
(\mathcal{A}f_{h})(x)=f_{h}''(x)-\frac{1}{3}xf_{h}(x)=\tilde{h}(x).
\]
Taking expectation in both sides, yields 
\[
\mathbb{E}\big((\mathcal{A}f_{h})(X)\big)=\mathbb{E}\big(\tilde{h}(X)\big)=\mathbb{E}\big(h(X)\big)-\mathbb{E}\big(h(Y)\big).
\]
Since $\mathbb{E}\big((\mathcal{A}f_{h})(X)\big)=0$, then 
\[
\mathbb{E}\big(h(X)\big)-\mathbb{E}\big(h(Y)\big)=0.
\]
Consequently, $X$ and $Y$ have the same law.
\end{proof}

\section{Stein's Characterization of the Symmetric $\hat{M}_{\nicefrac{1}{3}}$
Distribution}

\label{sec:Stein_MWright-sym}

The extended and normalized $M$-Wright function $M_{\beta}$, over
the negative real axis as an even function, becomes a probability
density in $\mathbb{R}$. We denote this extension by $\hat{M}_{\beta}$,
it is defined by
\[
\hat{M}_{\nicefrac{1}{3}}(x):=\frac{1}{2}M_{\nicefrac{1}{3}}(|x|):=\begin{cases}
\frac{1}{2}{\displaystyle M_{\nicefrac{1}{3}}(x),} & x\ge0,\\
\\
\frac{1}{2}{\displaystyle M_{\nicefrac{1}{3}}(-x),} & x<0.
\end{cases}
\]
\begin{figure}
\begin{centering}
\includegraphics[scale=0.75]{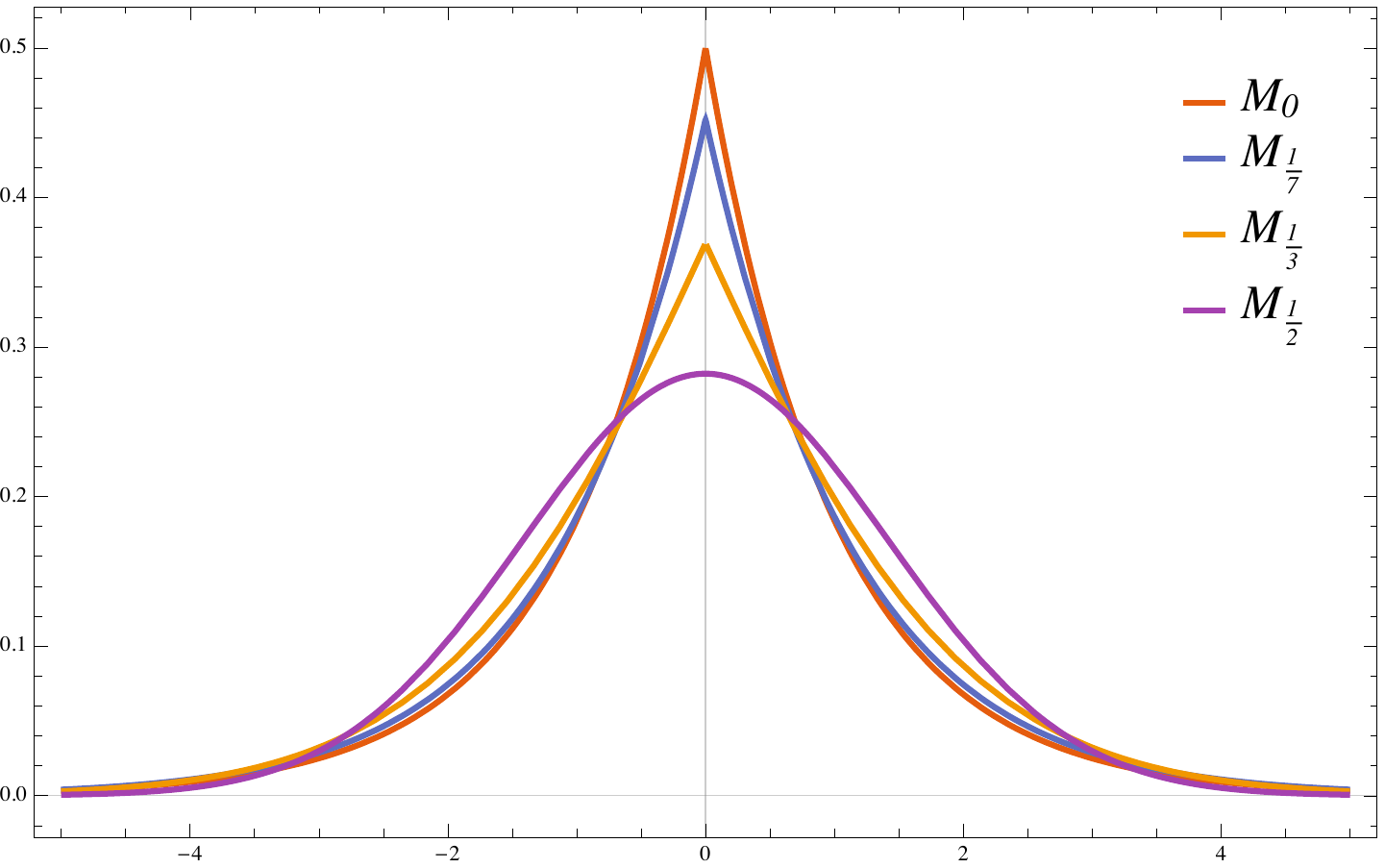}
\par\end{centering}
\caption{\label{fig:Plots-MWright}Plots of the densities $\hat{M}_{0}$, $\hat{M}_{\nicefrac{1}{7}}$, $\hat{M}_{\nicefrac{1}{3}}$ 
and $\hat{M}_{\nicefrac{1}{2}}$.}
\end{figure}
In this section we use the Stein method to characterize $\hat{M}_{\nicefrac{1}{3}}$.
To this end we define the space $\hat{\mathcal{D}}_{\nicefrac{1}{3}}$
by
\[
\hat{\mathcal{D}}_{\nicefrac{1}{3}}:=\{f\in C_{b}^{1}(\mathbb{R}),\;f''\in C_{b}(\mathbb{R}^{*})\,|\,f(0)=0\}.
\]
The Stein operator $\hat{\mathcal{A}}_{\nicefrac{1}{3}}$ on $C_{b}^{2}(\mathbb{R})$
is defined by 
\[
(\hat{\mathcal{A}}_{\nicefrac{1}{3}}f)(x):=f''(x)+\frac{1}{3}|x|f(x),\quad x\in\mathbb{R},
\]
and as Stein's equation associated to $\hat{\mathcal{A}}_{\nicefrac{1}{3}}$
\begin{equation}
y''(x)-\frac{1}{3}|x|y(x)=\hat{h}(x),\label{eq:Stein_equation_sym}
\end{equation}
where 
\begin{equation}
\hat{h}(x):=\big[h(x)-\mathbb{E}\big(h(Y)\big)\big]1\!\!1_{[0,\infty)}(x)+\big[h(x)-\mathbb{E}\big(h(-Y)\big)\big]1\!\!1_{(-\infty,0)}(x),\label{eq:h-hat}
\end{equation}
$h$ is a real-valued function and $Y$ has $M_{\nicefrac{1}{3}}$
distribution. Now we are ready to state the main result of this section.
\begin{thm}
\label{thm:charact_MW_sym}Let $X$ be a real-valued random variable
such that 
\begin{equation}
P(X\ge0)=P(X<0)=\frac{1}{2}.\label{eq:sym_assumption_X}
\end{equation}
Then $X$ follows the $\hat{M}_{\nicefrac{1}{3}}$ distribution if
and only if 
\[
\mathbb{E}\big((\hat{\mathcal{A}}_{\nicefrac{1}{3}}f)(X)\big)=0,\quad\forall f\in\hat{\mathcal{D}}_{\nicefrac{1}{3}}.
\]
\end{thm}

Before proving the above theorem we first show that the solution of
the Stein equation\ (\ref{eq:Stein_equation_sym}) belongs to $\hat{\mathcal{D}}_{\nicefrac{1}{3}}$.
This the contents of the following lemma.
\begin{lem}
\noindent \label{lem:Solution_Stein_sym}Let $h:\mathbb{R}\longrightarrow\mathbb{R}$
be a bounded continuous function. Then the function
\begin{align}
 & f_{\hat{h}}(x)\nonumber \\
 & =-3^{\nicefrac{1}{3}}\pi\left\{ \left[\mathrm{Ai}\left(\frac{x}{3^{\nicefrac{1}{3}}}\right)\int_{0}^{x}\mathrm{Bi}\left(\frac{t}{3^{\nicefrac{1}{3}}}\right)\hat{h}(t)\,dt+\mathrm{Bi}\left(\frac{x}{3^{\nicefrac{1}{3}}}\right)\int_{x}^{\infty}\mathrm{Ai}\left(\frac{t}{3^{\nicefrac{1}{3}}}\right)\hat{h}(t)\,dt\right]1\!\!1_{[0,\infty)}(x)\right.\nonumber \\
 & \left.+\left[\mathrm{Ai}\left(-\frac{x}{3^{\nicefrac{1}{3}}}\right)\int_{x}^{0}\mathrm{Bi}\left(-\frac{t}{3^{\nicefrac{1}{3}}}\right)\hat{h}(t)\,dt+\mathrm{Bi}\left(-\frac{x}{3^{\nicefrac{1}{3}}}\right)\int_{-\infty}^{x}\mathrm{Ai}\left(-\frac{t}{3^{\nicefrac{1}{3}}}\right)\hat{h}(t)\,dt\right]1\!\!1_{(-\infty,0)}(x)\right\} \label{eq:Stein_solution-sym-2}
\end{align}
 solves the Stein equation (\ref{eq:Stein_equation_sym}). In addition,
there exists non-negative constants $\hat{C}_{1}$, $\hat{C}_{2}$
and $\hat{C}_{3}$ such that $\|f_{\hat{h}}\|_{\infty}\le\hat{C}_{1}\|\hat{h}\|_{\infty}$,
$\|f_{\hat{h}}'\|_{\infty}\le\hat{C}_{2}\|\hat{h}\|_{\infty}$, $\|f_{\hat{h}}''\|_{\infty}\le\hat{C}_{3}\|\hat{h}\|_{\infty}$
and $f_{\hat{h}}$ belongs to $\hat{\mathcal{D}}_{\nicefrac{1}{3}}$.
\end{lem}

\begin{proof}
If follows from Lemma\ \ref{lem:Stein_solution} that the solution
$f_{\hat{h}}$ of equation (\ref{eq:Stein_equation_sym}) for $x\ge0$
is given by
\begin{equation}
f_{\hat{h}}(x)=-3^{\nicefrac{1}{3}}\pi\left[\mathrm{Ai}\left(\frac{x}{3^{\nicefrac{1}{3}}}\right)\int_{0}^{x}\mathrm{Bi}\left(\frac{t}{3^{\nicefrac{1}{3}}}\right)\hat{h}(t)\,dt+\mathrm{Bi}\left(\frac{x}{3^{\nicefrac{1}{3}}}\right)\int_{x}^{\infty}\mathrm{Ai}\left(\frac{t}{3^{\nicefrac{1}{3}}}\right)\hat{h}(t)\,dt\right].\label{eq:f_one}
\end{equation}
On the other hand, for $x<0$ the equation (\ref{eq:Stein_equation_sym})
turns into
\[
y''(x)+\frac{1}{3}xy(x)=\hat{h}(x).
\]
With a change of variables $u=-x$ the solution $f_{\hat{h}}$ of
the above differential equation is given by
\begin{equation}
f_{\hat{h}}(x)=-3^{\nicefrac{1}{3}}\pi\left[\mathrm{Ai}\left(-\frac{x}{3^{\nicefrac{1}{3}}}\right)\int_{x}^{0}\mathrm{Bi}\left(-\frac{t}{3^{\nicefrac{1}{3}}}\right)\hat{h}(t)\,dt+\mathrm{Bi}\left(-\frac{x}{3^{\nicefrac{1}{3}}}\right)\int_{-\infty}^{x}\mathrm{Ai}\left(-\frac{t}{3^{\nicefrac{1}{3}}}\right)\hat{h}(t)\,dt\right].\label{eq:f_two}
\end{equation}
Then the solution (\ref{eq:Stein_solution-sym-2}) follows putting
together both equations (\ref{eq:f_one}) and (\ref{eq:f_two}). It
remains to show that $f_{\hat{h}}\in\hat{\mathcal{D}}_{\nicefrac{1}{3}}$.
As $\hat{h}$ is a continuous function on $\mathbb{R}^{*}$, then
$f_{\hat{h}}\in C^{2}(\mathbb{R}^{*})$. At $x=0$, we have 
\[
f_{\hat{h}}(0)=-3^{-\nicefrac{1}{3}}\pi\mathrm{Bi}(0)\int_{0}^{\infty}M_{\nicefrac{1}{3}}(t)\big[h(t)-\mathbb{E}\big(h(Y)\big)\big]\,dt=0.
\]
and
\[
f_{\hat{h}}(0^{-})=-3^{-\nicefrac{1}{3}}\pi\mathrm{Bi}(0)\int_{-\infty}^{0}M_{\nicefrac{1}{3}}(-t)\big[h(t)-\mathbb{E}\big(h(-Y)\big)\big]\,dt=0,
\]
then $f_{\hat{h}}$ is continuous at $x=0$ and $f_{\hat{h}}(0)=0$.
On the other hand, 
\[
f_{\hat{h}}'(0)=-3^{-\nicefrac{1}{3}}\pi\mathrm{Bi}'(0)\int_{0}^{\infty}M_{\nicefrac{1}{3}}(t)\big[h(t)-\mathbb{E}\big(h(Y)\big)\big]\,dt=0
\]
and 
\[
f_{\hat{h}}'(0^{-})=-3^{\nicefrac{1}{3}}\pi\mathrm{Bi}'(0)\int_{-\infty}^{0}M_{\nicefrac{1}{3}}(-t)\big[h(t)-\mathbb{E}\big(h(-Y)\big)\big]\,dt=0.
\]
Therefore, $f_{\hat{h}}'$ is also continuous at $x=0$ and consequently
$f_{\hat{h}}\in C^{1}(\mathbb{R})$. It is easy to see that $f_{\hat{h}}''$
is continuous in $\mathbb{R}^{*}$ and 
\begin{eqnarray*}
f_{\hat{h}}''(0^{+}) & = & \hat{h}(0^{+})=h(0)-\mathbb{E}\big(h(Y)\big),\\
f_{\hat{h}}''(0^{-}) & = & \hat{h}(0^{-})=h(0)-\mathbb{E}\big(h(-Y)\big).
\end{eqnarray*}
Finally, we show that $f_{\hat{h}}$ and $f_{\hat{h}}'$ are bounded.
Moreover, as $f_{\hat{h}}$ satisfies the equation\ (\ref{eq:Stein_equation_sym})
it will follows that $f_{\hat{h}}''$ will be also bounded. On one
hand, for $x\ge0$, it follows from Lemma\ \ref{lem:Stein_solution}
that $f_{\hat{h}}$ given by equation\ (\ref{eq:f_one}) and its
derivatives are bounded, namely 
\begin{eqnarray*}
|f_{\hat{h}}(x)| & \le & 3^{\nicefrac{2}{3}}\pi\|\mathrm{Gi}\|_{\infty}\|\tilde{h}\|_{\infty},\\
|f_{\hat{h}}'(x)| & \leq & 3^{\nicefrac{1}{3}}\pi\|\mathrm{Gi}'\|_{\infty}\|\tilde{h}\|_{\infty},\\
|f_{\hat{h}}''(x)| & \le & \big(3^{-\nicefrac{2}{3}}\pi\sup_{x\in\mathbb{R}_{+}}|x\mathrm{Gi}(x)|+1\big)\|\tilde{h}\|_{\infty}.
\end{eqnarray*}
In a similar way for $x<0$, we estimate $f_{\hat{h}}$ given by equation\ (\ref{eq:f_two})
and its derivatives in order to show that $f_{\hat{h}}$ and its derivatives
are bounded. Thus, we conclude that $f_{\hat{h}}\in\hat{\mathcal{D}}_{\nicefrac{1}{3}}$.
\end{proof}
\begin{proof}[Proof of Proposition\ \ref{thm:charact_MW_sym}.]
\emph{Necessity}. Let $\hat{Y}$ be a random variable with $\hat{M}_{\nicefrac{1}{3}}$
distribution and $f\in\hat{\mathcal{D}}_{\nicefrac{1}{3}}$. Then,
it is clear that
\[
\int_{\mathbb{R}}f''(x)\hat{M}_{\nicefrac{1}{3}}(x)\,dx=\frac{1}{2}\left(\int_{-\infty}^{0}f''(x)M_{\nicefrac{1}{3}}(-x)\,dx+\int_{0}^{\infty}f''(x)M_{\nicefrac{1}{3}}(x)\,dx\right).
\]
Using the integration by parts and the asymptotic behavior of $M_{\nicefrac{1}{3}}$,
$M'_{\nicefrac{1}{3}}$ , see Appendix \ref{sec:special_functions},
we have
\begin{eqnarray*}
\int_{-\infty}^{0}f''(x)M_{\nicefrac{1}{3}}(-x)\,dx & = & f'(0)M_{\nicefrac{1}{3}}(0)+\int_{-\infty}^{0}f'(x)M'_{\nicefrac{1}{3}}(-x)\,dx\\
 & = & f'(0)M_{\nicefrac{1}{3}}(0)+\int_{-\infty}^{0}f(x)M''_{\nicefrac{1}{3}}(-x)\,dx.
\end{eqnarray*}
But we know that 
\[
M_{\nicefrac{1}{3}}''(-x)=-\frac{1}{3}xM_{\nicefrac{1}{3}}(-x),
\]
therefore we obtain 
\[
\int_{-\infty}^{0}f''(x)M_{\nicefrac{1}{3}}(-x)\,dx=f'(0)M_{\nicefrac{1}{3}}(0)-\frac{1}{3}\int_{-\infty}^{0}xf(x)M_{\nicefrac{1}{3}}(-x)\,dx
\]
equivalently 
\[
\int_{-\infty}^{0}\left(f''(x)+\frac{1}{3}xf(x)\right)M_{\nicefrac{1}{3}}(-x)\,dx=f'(0)M_{\nicefrac{1}{3}}(0).
\]
It follows from the proof of Theorem\ \ref{thm:charact_MW_dist}
that
\[
\int_{0}^{\infty}\left(f''(x)-\frac{1}{3}xf(x)\right)M_{\nicefrac{1}{3}}(x)\,dx=-f'(0)M_{\nicefrac{1}{3}}(0).
\]
Putting together, yields
\[
\int_{\mathbb{R}}\left(f''(x)+\frac{1}{3}|x|f(x)\right)\hat{M}_{\nicefrac{1}{3}}(x)\,dx=0.
\]
Therefore, we have
\[
\mathbb{E}\big((\hat{\mathcal{A}}_{\nicefrac{1}{3}}f)(\hat{Y})\big)=0.
\]

\noindent \emph{Sufficiency}. Let $h$ be a bounded continuous function.
By Lemma\ \ref{lem:Solution_Stein_sym}, the function $f_{\hat{h}}\in\hat{\mathcal{D}}_{\nicefrac{1}{3}}$
and satisfies the Stein equation: 
\[
(\hat{\mathcal{A}}_{\nicefrac{1}{3}}f_{\hat{h}})(x)=f_{\hat{h}}''(x)-\frac{1}{3}|x|f_{\hat{h}}(x)=\hat{h}(x).
\]
Then for $X$ a real-valued random variable satisfying the condition
(\ref{eq:sym_assumption_X}), we have 
\[
\mathbb{E}\big((\hat{\mathcal{A}}_{\nicefrac{1}{3}}f_{\hat{h}})(X)\big)=\mathbb{E}\big(\hat{h}(X)\big)=\mathbb{E}\big(h(X)\big)-\frac{1}{2}\big[\mathbb{E}\big(h(Y)\big)+\mathbb{E}\big(h(-Y)\big)\big].
\]
Since $\mathbb{E}((\hat{\mathcal{A}}_{\nicefrac{1}{3}}f_{h})(X))=0$
and $\mathbb{E}\big(h(\hat{Y})\big)=\frac{1}{2}\big[\mathbb{E}\big(h(Y)\big)+\mathbb{E}\big(h(-Y)\big)\big]$,
then 
\[
\mathbb{E}\big(h(X)\big)-\mathbb{E}\big(h(\hat{Y})\big)=0.
\]
Thus, $X$ and $\hat{Y}$ have the same law.

\end{proof}
\begin{rem}
\label{rem:Stein-sym-h}Assume that $X$ does not satisfies (\ref{eq:sym_assumption_X})
and take $h$ an even continuous function on $\mathbb{R}$, we have
\begin{enumerate}
\item $\hat{h}$ is continuous and consequently the solution $f_{\hat{h}}$
of the Stein equation (\ref{eq:Stein_solution-sym-2}) belongs to
$C_{b}^{2}(\mathbb{R})$, i.e., the space of bounded twice continuously
differentiable functions with bounded derivatives. 
\item In addition, we obtain that the random variables $|X|$ and $|\hat{Y}|$
have the same law but it is not enough to conclude that $X$ and $\hat{Y}$
have the same law.
\end{enumerate}
\end{rem}

\section{Conclusion}

\label{sec:outlook}We have characterized the distribution with density
$M_{\nicefrac{1}{3}}$ in $\mathbb{R}_{+}$ as well as its symmetrization
$\hat{M}_{\nicefrac{1}{3}}$ through the Stein method. The Stein operator
turns out to be the Airy equation, compare equations\ (\ref{eq:Airy_ODE})
and (\ref{eq:Stein-Operator}). The Stein equation corresponds to
the general inhomogeneous Airy equation\ (\ref{eq:Stein-Operator}).
As particular solution of the Stein equation we take a generalization
of the Scorer function $\mathrm{Gi}$, see equation\ (\ref{eq:ScorerGi_gen}).

The characterization of the class of distributions with density $M_{\nicefrac{1}{n}},$
$n=4,5,\ldots$ leads us to consider the Stein operator of the following
form
\begin{equation}
y^{(n-1)}-\frac{(-1)^{n}}{n}xy=0.\label{eq:hyper-Airy}
\end{equation}
This equation for $n\ge4$ is akin to the hyper-Airy differential
equation of order $n-1$, see \cite{Orszag1978}. In general for $q:=\nicefrac{1}{\beta}-1$,
$0<\beta<1$, the density $M_{\beta}$ satisfies the fractional differential
equation 
\begin{equation}
\frac{d^{q}}{dx^{q}}M_{\beta}(x)+\beta e^{\pm\nicefrac{i\pi}{\beta}}xM_{\beta}(x)=0.\label{eq:gene-hyper-Airy}
\end{equation}
In view of the above equation, the density $M_{\beta}$ is referred
in \cite{Mainardi1994} as a generalized hyper-Airy function. To find
a particular solution for the above equations\ (\ref{eq:hyper-Airy})
and (\ref{eq:gene-hyper-Airy}) seems to be a non trivial task and
here new ideas are needed.

\subsection*{Acknowledgments}

We would like to thank M.\ Röckner and M.\ Grothaus for helpful
discussions and comments during an earlier presentation of this work.
Financial support from FCT \textendash{} Funda{\c c\~a}o para a Ci{\^e}ncia
e a Tecnologia through the project UID/MAT/04674/2013 (CIMA) and the
Laboratory LIBMA form the University Cadi Ayyad Marrakech are grateful
acknowledged.

\appendix

\section{The Airy and Scorer functions}

\label{sec:special_functions}In this appendix we collect some properties
of the Airy and Scorer functions which are used throughout this paper.
We refer to the following books \cite{AS92,Olver2010,Olver1997,Vallee2004}
for more details and properties.

The second order homogeneous differential equation, known as Airy
equation 
\begin{equation}
y''-xy=0\label{eq:Airy_ode}
\end{equation}
has a pair of linear independent solutions $\mathrm{Ai}$ and $\mathrm{Bi}$,
called the Airy function of the first and second kind, respectively.
They are entire functions of $x$ with initial values 
\begin{eqnarray*}
\mathrm{Ai}(0) & = & \frac{1}{3^{\nicefrac{2}{3}}\Gamma\left(\frac{2}{3}\right)},\quad\mathrm{Ai}'(0)=-\frac{1}{3^{\nicefrac{1}{3}}\Gamma\left(\frac{1}{3}\right)},\\
\mathrm{Bi}(0) & = & \frac{1}{3^{\nicefrac{1}{6}}\Gamma\left(\frac{2}{3}\right)},\quad\mathrm{Bi}'(0)=\frac{3^{\nicefrac{1}{6}}}{\Gamma\left(\frac{1}{3}\right)},
\end{eqnarray*}
and their Wronskian is
\[
W\big(\mathrm{Ai}(x),\mathrm{Bi}(x)\big)=\frac{1}{\pi}.
\]
For the special case $x>0$, the Airy functions can be written in
terms of the Bessel functions 
\begin{eqnarray*}
\mathrm{Ai}(x) & = & \frac{1}{\pi}\sqrt{\frac{x}{3}}K_{\nicefrac{1}{3}}\left(\zeta\right),\qquad\zeta:=\frac{2}{3}x^{\nicefrac{3}{2}},\\
\mathrm{Bi}(x) & = & \sqrt{\frac{x}{3}}\left[I_{\nicefrac{1}{3}}\left(\zeta\right)+I_{-\nicefrac{1}{3}}\left(\zeta\right)\right],
\end{eqnarray*}
where $K_{\nicefrac{1}{3}}$ and $I_{\pm\nicefrac{1}{3}}$ are the
modified Bessel functions of the second and first second, respectively,
see \cite[Sec.~10.4]{AS92}. The asymptotic behavior of the Airy function
$\mathrm{Ai}$ and its derivative for $x\rightarrow\infty$
\begin{eqnarray}
\mathrm{Ai}(x)\sim\frac{1}{2}\pi^{-\nicefrac{1}{2}}x^{-\nicefrac{1}{4}}e^{-\zeta}, & \qquad & \mathrm{Ai}'(x)\sim-\frac{1}{2}\pi^{-\nicefrac{1}{2}}x^{\frac{1}{4}}e^{-\zeta}.\label{eq:AiryAi-asymptotic-positiv}
\end{eqnarray}
It is easy to see that the two linear independent solutions of the
general Airy equation 
\begin{equation}
y''-k^{2}xy=0\label{eq:Airy-ode-general}
\end{equation}
are $\mathrm{Ai}(k^{\nicefrac{2}{3}}x)$ and $\mathrm{Bi}(k^{\nicefrac{2}{3}}x)$.
The general solution of the inhomogeneous Airy equation 
\begin{equation}
y''-xy=\frac{1}{\pi}\label{eq:Airy_ode_non-homog_1}
\end{equation}
is
\[
y(x)=c_{1}\mathrm{Ai}(x)+c_{2}\mathrm{Bi}(x)+p(x),
\]
where $c_{1},c_{2}$ are arbitrary constants and $p(x)$ is any particular
solution of the equation\ (\ref{eq:Airy_ode_non-homog_1}). A standard
particular solution may be found by the method of variation of parameters
(see for example\ \cite{Collins2007}), called Scorer's function
$\mathrm{Gi}$ (also known as inhomogeneous Airy functions), given
by
\begin{equation}
\mathrm{Gi}(x):=\mathrm{Ai}(x)\int_{0}^{x}\mathrm{Bi}(t)\,dt+\mathrm{Bi}(x)\int_{x}^{\infty}\mathrm{Ai}(t)\,dt.\label{eq:ScorerGi-function}
\end{equation}
The Scorer function $\mathrm{Gi}$ is an entire bounded function of
$x.$ Its asymptotic as $x\rightarrow\infty$ is given by
\begin{equation}
\mathrm{Gi}(x)\sim\frac{1}{\pi x}\quad\Rightarrow\quad x\mathrm{Gi}(x)\sim\frac{1}{\pi}.\label{eq:ScorerGi_asymptotic}
\end{equation}
The derivative of the Scorer function $\mathrm{Gi}$ is given by
\[
\mathrm{Gi}'(x)=\frac{1}{3}\mathrm{Bi}'(x)+\int_{0}^{x}\big[\mathrm{Ai}'(x)\mathrm{Bi}(t)-\mathrm{Ai}(t)\mathrm{Bi}'(t)\big]\,dt
\]
and the following asymptotic as $x\rightarrow\infty$ holds
\begin{equation}
\mathrm{Gi}'(x)\sim-\frac{1}{\pi x^{2}},\quad x\in\mathbb{R}_{+}.\label{eq:Scorer_deriv_asymp}
\end{equation}

In Sections\ \ref{sec:Stein-MWright} and \ref{sec:Stein_MWright-sym}
we use a generalization of equation\ (\ref{eq:Airy_ode_non-homog_1}),
namely
\begin{equation}
y''-k^{2}xy=f(x),\label{eq:Airy_ode_non-homog_2}
\end{equation}
with $k=3^{-\nicefrac{1}{2}}$ and $f$ a bounded measurable function.
Using the above scheme, the general solution of equation\ (\ref{eq:Airy_ode_non-homog_2})
is given by
\[
y(x)=c_{1}\mathrm{Ai}(k^{\nicefrac{2}{3}}x)+c_{2}\mathrm{Bi}(k^{\nicefrac{2}{3}}x)+q(x),
\]
where, as before $c_{1},c_{2}$ are arbitrary constants and $q(x)$
is any particular solution of the equation\ (\ref{eq:Airy_ode_non-homog_2}).
One convenient choice of $q(x)$ is 
\begin{equation}
q(x)=-k^{\nicefrac{2}{3}}\pi\left(\mathrm{Ai}(k^{-\nicefrac{2}{3}}x)\int_{0}^{x}\mathrm{Bi}(k^{-\nicefrac{2}{3}}t)f(t)\,dt+\mathrm{Bi}(k^{-\nicefrac{2}{3}}x)\int_{x}^{\infty}\mathrm{Ai}(k^{-\nicefrac{2}{3}}t)f(t)\,dt\right).\label{eq:ScorerGi_gen}
\end{equation}

\bibliographystyle{alpha}

\end{document}